\documentclass{article}%
\usepackage{amsfonts}
\usepackage{amssymb}
\usepackage{makeidx}
\usepackage{amsmath}
\usepackage{graphicx}%
\providecommand{\U}[1]{\protect\rule{.1in}{.1in}}
\newtheorem{theorem}{Theorem}

\newtheorem{condition}[theorem]{Condition}

\newtheorem{corollary}[theorem]{Corollary}

\newtheorem{definition}[theorem]{Definition}

\newtheorem{lemma}[theorem]{Lemma}

\newtheorem{proposition}[theorem]{Proposition}

\newenvironment{proof}[1][Proof]{\noindent\textbf{#1.} }{\ \rule{0.5em}{0.5em}}
\begin{document}

\title{An implicit function theorem for non-smooth maps between Fr\'{e}chet spaces.}
\author{Ivar Ekeland, Eric S\'{e}r\'{e}\\CEREMADE, Universit\'{e} Paris-Dauphine, 75016 Paris, France}
\date{November 2013}
\maketitle

\section{Introduction}

In this paper, we prove a "hard" inverse function theorem, that is, an inverse
function theorem for maps $F$ which lose derivatives:\ $F\left(  u\right)  $
is less regular than $u$. Such theorems\ have a long history, starting with
Kolmogorov in the Soviet Union (\cite{Arnold}, \cite{Arnold2}, \cite{Arnold3})
and Nash in the United States (\cite{Nash}), and it would be impossible, in
such a short paper, to give a full account of the developments which have
occured since.  Important contributions have been made since by H\"{o}rmander,
Zehnder, Mather, Sergeraert, Tougeron, Hamilton, Hermann, Craig, Dacorogna,
Bourgain, Berti and Bolle, and lately by Villani and Mouhot. However, all the
results which we are aware of require the function $F$ to be inverted to be at
least $C^{2}$; in the Kolmogorov-Arnol'd-Moser tradition, for instance, one
uses the fast convergence of Newton's method to overcome the loss of
derivatives. In contrast, we make no smoothness assumption on $F$, only that
is continuous and G\^{a}teaux-differentiable.

We will overcome the loss of derivatives by using a new version of the "soft"
inverse function theorem (between Banach spaces), the proof of which is given
in \cite{IE3}, namely:

\begin{theorem}
\label{Ivar} Let $X$ and $Y$ be Banach spaces, with respective norms
$\left\Vert x\right\Vert $ and $\left\Vert y\right\Vert ^{\prime}$. Let
$f:X\rightarrow Y$ be continuous and G\^{a}teaux-differentiable, with
$f\left(  0\right)  =0$. Assume that the derivative $Df\left(  x\right)  $ has
a right-inverse $[Df(x)]_{r}^{-1}$, uniformly bounded in a neighbourhood of
$0$:%
\begin{align*}
Df\left(  x\right)  [Df(x)]_{r}^{-1}  &  h=h\ \\
\sup\left\{  \left\Vert [Df(x)]_{r}^{-1}\right\Vert \ |\ \left\Vert
x\right\Vert \leq R\right\}   &  <m
\end{align*}

Then, for every $y\in Y$ with $\left\Vert y\right\Vert ^{\prime}\leq Rm^{-1}%
~$\ there is some $\bar{x}\in X$ with $\left\Vert \bar{x}\right\Vert \leq R$,
such that $f\left(  \bar{x}\right)  =\bar{y}$ and $\left\Vert \bar
{x}\right\Vert \leq m\left\Vert \bar{y}\right\Vert ^{\prime}$.
\end{theorem}

As we just said, we will make no attempt to review the literature on hard
inverse function theorems; see the survey by Hamilton \cite{Hamilton} for an
account up to 1982. We have drawn inspiration from the version in
\cite{Alinhac}, which itself is inspired from H\"{o}rmander's result
\cite{Hormander}. We have also learned much from the work on the nonlinear
wave equation by Berti and Bolle, \cite{BB1}, \cite{BB2}, \cite{BB},
\cite{BBP} whom we thank for extensive discussions.

In the section \ref{1}, we state our main result, Theorem 1, and we derive it
from an approximation procedure, which is described in Theorem 2. We also
given some variants of Theorem 1, for instance an implicit function theorem
and we describe a particular case when we can gain some regularity. Theorem 2
is proved in section \ref{2}, and the theoretical part is thus complete. The
next two sections are devoted to applications. Section \ref{3} revisits the
classical isometric imbedding problem, which was the purpose for Nash's
original work. This is somewhat academic, since it known now that it can be
treated without resorting to a hard inverse function theorem (see
\cite{Gunther}), but it gives us the opportunity to show on a simple example
how our result improves, for instance, on those of Moser \cite{Moser}.

\section{The setting\label{1}}

\subsection{The spaces}

Let $(X_{s},\,\Vert\cdot\Vert_{s})_{s\geq0}$ be a scale of Banach spaces:
\[
0\leq s_{1}\leq s_{2}\Longrightarrow(X_{s_{2}}\subset X_{s_{1}}\text{
\ and\ \ }\Vert\cdot\Vert_{s_{1}}\leq\Vert\cdot\Vert_{s_{2}})
\]

We shall assume that there exists a sequence of projectors $\Pi_{N}%
:\,X_{0}\rightarrow E_{N}$ where $E_{N}\subset\bigcap_{s\geq0}X_{s}$ is the
range of $\Pi_{N}$, with $\Pi_{0}=0$, $E_{N}\subset E_{N+1}$ and
$\bigcup_{N\geq1}E_{N}$ is dense in each space $X_{s}$ for the norm
$\Vert\cdot\Vert_{s}$. We assume that for any finite constant $A$ there is a
constant $C_{1}^{A}>0$ such that, for all nonnegative numbers $s,\,d$
satisfying $s+d\leq A$:%

\begin{equation}
\Vert\Pi_{N}u\Vert_{s+d}\leq C_{1}^{A}N^{d}\Vert u\Vert_{s} \label{loss}%
\end{equation}

\begin{equation}
\Vert(1-\Pi_{N})u\Vert_{s}\leq C_{1}^{A}N^{-d}\Vert u\Vert_{s+d} \label{gain}%
\end{equation}

Note that these properties imply some interpolation inequalities, for $0\leq
t\leq1$ and $0\leq s_{1}\,,\,s_{2}\leq A$ , and for a new constant
$C_{A}^{(2)}$ (see e.g. \cite{BBP}):%
\begin{equation}
\Vert x\Vert_{ts_{1}+(1-t)s_{2}}\leq C_{2}^{A}\Vert x\Vert_{s_{1}}^{t}\Vert
x\Vert_{s_{2}}^{1-t}\;. \label{inter}%
\end{equation}

If all these properties are satisfied, we shall say that the scale $\left(
X_{s}\right)  ,\ s\geq0$, is \emph{regular}, and we shall refer to the
$\Pi_{N}$ as \emph{smoothing operators}. Let $(Y_{s},\,\Vert\cdot\Vert
_{s}^{\prime})_{s\geq0}$ be another regular scale of Banach spaces. We shall
denote by $\Pi_{N}^{\prime}:Y_{0}\rightarrow E_{N}^{\prime}\subset
\bigcap_{s\geq0}Y_{s}$ the smoothing operators. In the sequel, $B_{s}(R)$
(resp. $B_{s}^{\prime}(R)$) will denote the open ball of center $0$ and radius
$R$ in $X_{s}$ (resp. $Y_{s}$)

\subsection{The map}

Recall that a map $F:X\rightarrow Y$, where $X$ and $Y$ are Banach spaces, is
\emph{G\^{a}teaux-differentiable} at $x$ if there is a linear map $DF\left(
x\right)  :X\rightarrow Y$ such that:%
\[
\forall h\in X,\ \ \lim_{t\rightarrow0}\frac{1}{t}\left[  F\left(
x+th\right)  -F\left(  x\right)  \right]  =DF\left(  x\right)  h
\]

In the following, $R>0$ and $S>0$ are prescribed, with possibly $S=\infty$

\begin{definition}
We shall say that $F:\,B_{0}(R)\rightarrow Y_{0}$ is \emph{roughly tame} with
loss of regularity $\mu$ if:

\begin{description}
\item[(a)] $F$ is continuous and G\^{a}teaux-differentiable from $B_{0}(R)\cap
X_{s}$ to $Y_{s}$ for any $s\in\lbrack0,S)$.

\item[(b)] For any $A\in\lbrack0,S)$ there is a finite constant $K_{A}$ such
that, for all $s<A$ and $x\in B_{0}\left(  R\right)  $:
\begin{equation}
\forall h\in X,\ \ \Vert DF(x)h\Vert_{s}^{\prime}\leq K^{A}(\Vert h\Vert
_{s}+\Vert x\Vert_{s}\Vert h\Vert_{0})\ \label{tame}%
\end{equation}

\item[(c)] For $x\in B_{0}(R)\cap E_{N}$, the linear maps $L_{N}%
(x):\,E_{N}\rightarrow E_{N}^{\prime}$ defined by $L_{N}=\Pi_{N}^{\prime
}DF\left(  x\right)  |_{E_{N}}$ have a right-inverse, denoted by
$[L_{N}(x)]_{r}^{-1}$. There is a constant $\mu>0\,$ and, for any $A\in
\lbrack0,S)$, a positive constant $\gamma_{A}$, such that, for all $s<A$ and
$x\in B_{0}\left(  R\right)  $ we have:%
\begin{equation}
\forall k\in E_{N}^{\prime},\ \ \Vert\lbrack L_{N}(x)]_{r}^{-1}k\Vert_{s}%
\leq\frac{1}{\gamma^{A}}N^{\mu}(\Vert k\Vert_{s}^{\prime}+\Vert x\Vert
_{s}\Vert k\Vert_{0}^{\prime}) \label{tameinverse}%
\end{equation}

\end{description}
\end{definition}

In our assumptions, there is no regularity loss between $x$ and $F(x)$, or
more exactly the regularity loss, if there is one, has been absorbed by
translating the indexation of the spaces $Y_{s}$. The number $S$ represents
the maximum regularity available, and the constant $\mu$ may be interpreted as
the loss of derivatives incurred when solving the linearized equation
$L_{N}(x)h=k$. Note that we need $\mu<S$ to start the process.

When trying to solve $F(x)=y$, it is thus natural to assume that $y-F\left(
0\right)  $ is small in $Y_{\mu}$and look for $x$ in $X_{0}$. This was done in
\cite{IE3} by assuming that $DF(x)$ has a right inverse which satisfied
estimates similar to (\ref{tame}) and (\ref{tameinverse}), but which were
independent of the base point $x$. In the present work, since the tame
estimates depend on $x$ with loss of regularity, we will have to assume that
$y$ is small in a more regular space $Y_{\delta}$, with $\delta>\mu$.\ 

\subsection{The main result}

We will need an assumption relating $\mu$, $\delta$ and $S$ with $\mu<\delta$
and $S>\mu.$ Here it is:

\begin{condition}
\label{c0}There is some $\kappa$ such that:
\end{condition}

\begin{align}
&  1<\kappa<2\text{ \ and\ \ }\min\{\kappa^{2},\,\kappa+1\}\mu<\delta
\label{delta}\\
&  \frac{\kappa^{2}}{\kappa-1}\mu<S \label{S}%
\end{align}

Inequality (\ref{S}) imposes $S>4\mu$, since 4 is the minimum value of
$\frac{\kappa^{2}}{\kappa-1}$, attained when $\kappa=2$. On the other hand,
$\min\{\kappa^{2},\,\kappa+1\}$ is an increasing function of $\kappa$, which
coincides with $\kappa^{2}$ when $\kappa\leq\left(  1+\sqrt{5}\right)  /2$ and
coincides with $\kappa+1$ when $\kappa>\left(  1+\sqrt{5}\right)  /2$. So
inequality (\ref{delta}) imposes $\delta>1$, attained when $\kappa=1$.

Let us represent condition (\ref{c0}) geometrically. Define a real function
$\varphi$ on $(0,\ 3]$ by:%
\begin{equation}
\varphi\left(  x\right)  =\left\{
\begin{array}
[c]{ccc}%
\frac{1}{2}x\left(  1-\sqrt{1-\frac{4}{x}}\right)  +1 & \text{if} &
4<x\leq\frac{3+\sqrt{5}}{\sqrt{5}-1}\\
\frac{x^{2}}{4}\left(  1-\sqrt{1-\frac{4}{x}}\right)  ^{2} & \text{if} &
x\geq\frac{3+\sqrt{5}}{\sqrt{5}-1}%
\end{array}
\right.  \label{x1}%
\end{equation}

\begin{proposition}
\label{c2} $\left(  \mu,\delta,S\right)  \in R_{+}^{2}$ satisfies condition
(\ref{c0}) if and only if $\frac{\delta}{\mu}\geq3$ or $\frac{\delta}{\mu}%
\leq3$ and $\frac{\delta}{\mu}\geq\varphi\left(  \frac{S}{\mu}\right)  $
\end{proposition}

\begin{proof}
Follows immediately from inverting formulas (\ref{delta}) and (\ref{S}).
\end{proof}

We have represented the admissible region for $\left(  \frac{S}{\mu}%
,\frac{\delta}{\mu}\right)  $ on Figure 1: it is the region $\Omega$ above the
curve.%
\begin{center}
\fbox{\includegraphics[
natheight=3.000000in,
natwidth=4.499600in,
height=3in,
width=4.4996in
]%
{C:/Users/ekeland/Dropbox/Sere/graphics/perte-minimale-10__1.png}%
}\\
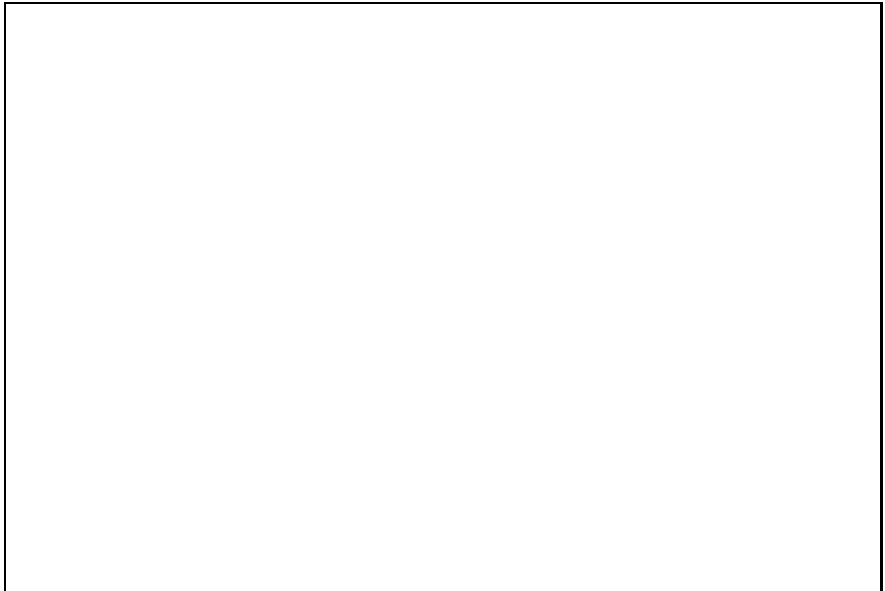
Figure 1:\ The function $\varphi\left(  x\right)  $%
\end{center}

The parameter $\kappa$ decreases from $\kappa=2$ (corresponding to $S/\mu
=3$)\ to $\kappa=1$ (corresponding to $S/\mu\rightarrow\infty$) along the
curve. Note the kink at$\ S/\mu=\frac{3+\sqrt{5}}{\sqrt{5}-1}$, $\delta
/\mu=\frac{1}{2}\left(  3+\sqrt{5}\right)  \ $corresponding to $\kappa
=\frac{1}{2}\left(  1+\sqrt{5}\right)  $ (the golden ratio). In the sequel, we
will separate the case $\kappa\leq\frac{1+\sqrt{5}}{2}$ (to the right) from
the case $\kappa\geq\frac{1+\sqrt{5}}{2}$ (to the left):%

\[%
\begin{array}
[c]{ccc}%
1<\kappa\leq\frac{1+\sqrt{5}}{2} & \frac{\delta}{\mu}>\kappa^{2}\geq1 &
\frac{S}{\mu}>\frac{\kappa^{2}}{\kappa-1}\geq\frac{3+\sqrt{5}}{\sqrt{5}-1}\\
\frac{1+\sqrt{5}}{2}\leq\kappa<2 & \frac{\delta}{\mu}>\kappa+1\geq
\frac{3+\sqrt{5}}{2} & \frac{S}{\mu}>\frac{\kappa^{2}}{\kappa-1}\geq4
\end{array}
\]

\begin{theorem}
\label{inverse} Assume $F:B_{0}\left(  R\right)  \cap X_{s}\rightarrow Y_{s}$,
$0\leq s<S$, is roughly tame with loss of regularity $\mu$. Suppose $F\left(
0\right)  =0$. Let $\delta>0$ and $\alpha>0$ be such that
\begin{align}
\frac{\delta}{\mu}  &  >\varphi\left(  \frac{S}{\mu}\right) \label{n0}\\
\frac{\alpha}{\mu}  &  <\min\left\{  \frac{\delta}{\mu}-\varphi\left(
\frac{S}{\mu}\right)  ,\ \frac{S}{\mu}-\varphi^{-1}\left(  \frac{\delta}{\mu
}\right)  \right\}  \label{n3}%
\end{align}

Then one can find $\rho\,>0$ and$\,$\ $C>0$ such that, for any $y\in
Y_{\delta}$ with $\left\Vert y\right\Vert _{\delta}^{\prime}\leq\rho$ , there
some $x\in X_{\alpha}$ such that:
\begin{align*}
F\left(  x\right)   &  =y\\
\left\Vert x\right\Vert _{0}  &  \leq1\\
\left\Vert x\right\Vert _{\alpha}  &  \leq C\left\Vert y\right\Vert _{\delta}%
\end{align*}

\end{theorem}

It follows that the map $F$ sends $X_{0}$ into $Y_{0}$, while $F^{-1}$ sends
$Y_{\delta}$ into $X_{0}$. This parallels the situation with the linearized
operator $DF\left(  0\right)  $, which sends $X_{0}$ into $Y_{0}$ while
$DF\left(  0\right)  ^{-1}$ sends $Y_{\mu}$ into $X_{0}$, with $\mu<\delta$.
More precisely, we have:\ 

\begin{itemize}
\item if $F\left(  \bar{x}\right)  =\bar{y}$, then $F\left(  X_{0}\right)
\ $\ contains some $\delta$-neighbourhood of $\bar{y}$

\item if $F\left(  \bar{x}\right)  =\bar{y}$, then $F^{-1}\left(  \bar
{y}+Y_{\delta}\right)  $ contains some $\alpha$-neighbourhood of $\bar{x}$
\end{itemize}

$S$ is the maximal regularity on $x$\ and $\delta$ is the minimal regularity
on $y$ that we will need in the approximation procedure, bearing in mind that
$F\left(  0\right)  =0$ (see Corollary \ref{c11} below for the case when
$F\left(  0\right)  =\bar{y}\neq0$ has fininte regularity). Note that we may
have $\delta>S$: this simply means that the right-hand side $y$ is more
regular than the sequence of approximate solutions $x_{n}$ that we will construct.

$\alpha>0$ is the regularity of the solution $x$. Note the significance of
\ref{n0}) and (\ref{n3}) take together. The inequality $\frac{\delta}{\mu
}>\varphi\left(  \frac{S}{\mu}\right)  $ tells us that the right-hand side $y$
is more regular than needed (or, alternatively, that the full range of $S$ has
not been used), and this "excess regularity", measured by the difference
$\frac{\delta}{\mu}-\varphi\left(  \frac{S}{\mu}\right)  $ (or, alternatively,
$\frac{S}{\mu}-\varphi^{-1}\left(  \frac{\delta}{\mu}\right)  $ ) can be
diverted to $x$. For instance, if $S/\mu\rightarrow\infty$, the total loss of
regularity $\delta-\alpha$ between $y$ and $x$ satisfies%
\[
\delta-\alpha>\mu\varphi\left(  \frac{S}{\mu}\right)
\]
and can be made as close to $\mu$ as one wishes:\ we can start the iterative
procedure $x_{n}$ from a very regular initial point. However, we lose control
on $\rho$ (which goes to zero) and $C$ (which goes to infinity). On the other
hand, when $\delta/\mu>3$, that is, when we are not worried about the loss of
regularity, then $S$ can be any number larger than $\mu$, that is, we need
very little regularity to start with.

We now go from the case $F\left(  0\right)  =0$ to the case $F\left(  \bar
{x}\right)  =\bar{y}$. There is some subtelty there because $0$ belong to all
the $X_{s}$, while $\bar{y}$ does not, and puts adittional limits to the
regularity. We shall say that $F$ is roughly tame at $\bar{x}$ if $F\left(
x-\bar{x}\right)  $ is roughly tame at $0$.

\begin{corollary}
\label{c11} Suppose $\bar{x}\in X_{S_{1}}$, $\bar{y}\in Y_{S_{2}}$ and
$F\left(  \bar{x}\right)  =\bar{y}$. Assume $F\left(  x\right)  $ sends
$X_{s}\cap B_{0}\left(  \bar{x},R\right)  $ into $Y_{s}$ for every $s\geq0$,
and is roughly tame at $\bar{x}$ with loss of regularity $\mu$. Set
$S=\min\left\{  S_{1},S_{2}\right\}  $. Let $\delta$ and $\alpha$ satisfy
(\ref{n0}) and (\ref{n3}). Then one can find $\rho\,>0$ and$\,$\ $C>0$ such
that, for any $y$ with $\left\Vert \bar{y}-y\right\Vert _{\delta}^{\prime}%
\leq\rho$ , there is a solution $x$ of the equation $F(x)=y,$ with $\left\Vert
x-\bar{x}\right\Vert _{0}\leq1$ and $\Vert x-\bar{x}\Vert_{\alpha}\leq C\Vert
y-\bar{y}\Vert_{\delta}^{\prime}\,.$
\end{corollary}

\begin{proof}
Consider the map $\Phi\left(  x\right)  :=F\left(  x+\bar{x}\right)  -\bar{y}%
$. It is roughly tame, with $F\left(  0\right)  =0$, and we can apply the
preceding Theorem with $S=\min\left\{  S_{1},S_{2}\right\}  $. The result follows
\end{proof}

We now deduce an implicit function theorem. Let $V$ be a Banach space and let
$F:B\left(  R,X_{0}\times V\right)  \cap\left(  X_{s}\times V\right)
\rightarrow Y_{s}$, $0\leq s<S$ satisfy the following:

\begin{definition}
\begin{description}
\item[(a')] $F$ is continuous and G\^{a}teaux-differentiable for any
$s\in\lbrack0,S)$. We write:%
\[
DF\left(  x,v\right)  =\left(  D_{x}F\left(  x,v\right)  ,\ D_{v}F\left(
x,v\right)  \right)
\]

\item[(b')] For any $A\in\lbrack0,S)$ there is a finite constant $K_{A}$ such
that, for all $s<A$ and $\left(  x,v\right)  \in B\left(  R,X_{0}\times
V\right)  $:
\[
\forall h\in X,\ \ \Vert DF(x,v)h\Vert_{s}^{\prime}\leq K^{A}(\Vert h\Vert
_{s}+\Vert x\Vert_{s}\Vert h\Vert_{0})\
\]

\item[(c')] For $x\in\left(  x,v\right)  \in B\left(  R,X_{0}\times V\right)
\cap\left(  E_{N}\times V\right)  $, the linear maps $L_{N}(x,v):\,E_{N}%
\rightarrow E_{N}^{\prime}$ defined by $L_{N}=\Pi_{N}^{\prime}D_{x}F\left(
x,v\right)  |_{E_{N}}$ have a right-inverse, denoted by $[L_{N}(x,v)]_{r}%
^{-1}$. There is a constant $\mu>0\,$ and, for any $A\in\lbrack0,S)$, a
positive constant $\gamma_{A}$, such that, for all $s<A$ and $\left(
x,v\right)  \in B\left(  R,X_{0}\times V\right)  $ we have:%
\[
\forall k\in E_{N}^{\prime},\ \ \Vert\lbrack L_{N}(x,v)]_{r}^{-1}k\Vert
_{s}\leq\frac{1}{\gamma^{A}}N^{\mu}(\Vert k\Vert_{s}^{\prime}+\Vert x\Vert
_{s}\Vert k\Vert_{0}^{\prime})
\]

\end{description}
\end{definition}

\begin{corollary}
\label{c21} Assume (a'), (b'), (c' ) are satisfied and $F\left(  0,0\right)
=0$ . Take any $\alpha$ with $0<\alpha<S-4\mu$. Then one can find $\rho\,>0$
and$\,$\ $C>0$ such that, for any $v$ with $\left\Vert v\right\Vert \leq\rho$
, there is a some $x$ such that:%
\begin{align*}
F\left(  x,v\right)   &  =0\\
\left\Vert x\right\Vert _{0}  &  \leq1\\
\left\Vert x\right\Vert _{\alpha}  &  \leq C\left\Vert v\right\Vert
\end{align*}

\end{corollary}

\begin{proof}
Consider the Banach scale $X_{s}\times V$ and $Y_{s}\times V$ with the natural
norms. Consider the map $\Phi\left(  x,v\right)  =\left(  F\left(  x,v\right)
,\ v\right)  $ from $X_{s}\times V$, $0\leq s<S$, into $Y_{s}\times V$. It is
roughly tame with $\Phi\left(  0,0\right)  =\left(  0,0\right)  $ and we can
apply the preceding Theorem with $\delta=\infty$. Condition (\ref{n3})
becomes
\[
\frac{\alpha}{\mu}<\frac{S}{\mu}-4
\]

\end{proof}

\subsection{A particular case}

In the case when $F\left(  x\right)  =Ax+G\left(  x\right)  $ where $A$ is
linear, we can improve the regularity.

\begin{proposition}
\label{Prop} Suppose $F\left(  x\right)  =Ax+G\left(  x\right)  $, where
$A:X_{s+\nu}\rightarrow Y_{s}$ is a continuous linear operator, independent of
the base point $x$, and $G$ satisfies $G\left(  0\right)  =0$. Suppose
moreover that:

\begin{description}
\item[(a)] $G$ is continuous and G\^{a}teaux-differentiable from $B_{0}(R)\cap
X_{s}$ to $Y_{s}$ for any $s\in\lbrack0,\ S)$.

\item[(b)] For any $A\in\lbrack0,\ S)$ there is a finite constant $K_{A}$ such
that, for all $s\leq A$ and $x\in B_{0}\left(  R\right)  $:
\[
\forall h\in X,\ \ \Vert DG(x)h\Vert_{s}^{\prime}\leq K^{A}(\Vert h\Vert
_{s}+\Vert x\Vert_{s}\Vert h\Vert_{0})\
\]

\item[(c)] For $x\in B_{0}(R)\cap E_{N}$, the linear maps $L_{N}%
(x):\,E_{N}\rightarrow E_{N}^{\prime}$ have a right-inverse, denoted by
$[L_{N}(x)]_{r}^{-1}$. There is a constant $\mu>0\,$ and, for any $A\in
\lbrack0,\ S)$, a positive constant $\gamma_{A}$, such that, for all
$s\in\lbrack0,\ S)$ and $x\in B_{0}\left(  R\right)  $ we have:%
\[
\forall k\in E_{N}^{\prime},\ \ \Vert\lbrack L_{N}(x)]_{r}^{-1}k\Vert_{s}%
\leq\frac{1}{\gamma^{A}}N^{\mu}(\Vert k\Vert_{s}^{\prime}+\Vert x\Vert
_{s}\Vert k\Vert_{0}^{\prime})
\]

\item[(d)] The $E_{N}$ are $A$-invariant.
\end{description}

Let $\delta>0$ and $\alpha>0$ be such that
\begin{align*}
\frac{\delta}{\mu}  &  >\varphi\left(  \frac{S}{\mu}\right) \\
\frac{\alpha}{\mu}  &  <\min\left\{  \frac{\delta}{\mu}-\varphi\left(
\frac{S}{\mu}\right)  ,\ \frac{S}{\mu}-\varphi^{-1}\left(  \frac{\delta}{\mu
}\right)  \right\}
\end{align*}
. Then one can find $\rho\,>0$ and$\,$\ $C>0$ such that, for any $y\in
Y_{\delta}$ with $\left\Vert y\right\Vert _{\delta}^{\prime}\leq\rho$ , there
some $x\in X_{\alpha}$ such that:
\begin{align*}
F\left(  x\right)   &  =y\\
\left\Vert x\right\Vert _{0}  &  \leq1\\
\left\Vert x\right\Vert _{\alpha}  &  \leq C\left\Vert y\right\Vert _{\delta}%
\end{align*}

\end{proposition}

In this situation, a direct application of Theorem \ref{inverse} would give a
loss or regularity of $\mu+\nu$. Proposition \ref{Prop} tells us that the
$\mu$ is enough: the loss of regularity due to the linear part can be circumventend.

\subsection{The approximating sequence{}}

Theorem \ref{inverse} is proved by an approximation procedure: we construct by
induction a sequence $x_{n}$ having certain properties, and we show that it
converges to the desired solution. We now describe that sequence, and give the
proof of convergence. The actual construction of the sequence is postponed to
the next section.

Given an integer $N$ and a real number $\alpha>1$, we shall denote by
$E\left[  N^{\alpha}\right]  $ the integer part of $N^{\alpha}$:
\[
E\left[  N^{\alpha}\right]  \leq N^{\alpha}<E\left[  N^{\alpha}\right]  +1
\]

\begin{theorem}
\label{iteration} Assume that $\mu,\delta,S$ and $\kappa$ satisfy condition
\ref{c1}. Choose $\sigma\ $and $\beta$ such that:%
\begin{equation}
\frac{\kappa^{2}}{\kappa-1}\mu<\kappa\beta<\sigma<S \label{commun}%
\end{equation}

Impose, moreover:

\begin{itemize}
\item For $1<\kappa\leq\frac{1+\sqrt{5}}{2}$
\begin{equation}
\kappa\beta>\kappa\mu+\sigma-\frac{\delta}{\kappa} \label{beta1}%
\end{equation}

\item For $\frac{1+\sqrt{5}}{2}\leq\kappa<2$%
\begin{equation}
\beta>\mu+\sigma-\delta\label{beta2}%
\end{equation}

\end{itemize}

Then one can find $N_{0}\geq2$, $\rho>0$ and $\,c>0$ such that, for any $y\in
Y$ with $\left\Vert y\right\Vert _{\delta}^{\prime}\leq\rho$, there are
sequences $(x_{n})_{n\geq1}$ in $B_{0}(1)$ and $N_{n}:=N_{0}^{\left(  \kappa
n\right)  }\ $satisfying:

\begin{itemize}
\item For $1<\kappa\leq\frac{1+\sqrt{5}}{2}$,%
\begin{equation}
\Pi_{N_{n}}^{\prime}F(x_{n})=\Pi_{N_{n-1}}^{\prime}y\;\mathrm{and}\;x_{n}\in
E_{N_{n}} \label{projected}%
\end{equation}

\item For $\frac{1+\sqrt{5}}{2}\leq\kappa<2$%
\begin{equation}
\Pi_{N_{n}}^{\prime}F(x_{n})=\Pi_{N_{n}}^{\prime}y\;\mathrm{and}\;x_{n}\in
E_{N_{n}} \label{projected'}%
\end{equation}

\end{itemize}

And in both cases:%
\begin{align}
\Vert x_{1}\Vert_{0}  &  \leq cN_{1}^{\mu}\Vert y\Vert_{\delta}^{\prime
}\;\mathrm{and}\;\Vert x_{n+1}-x_{n}\Vert_{0}\leq c\,N_{n}^{\kappa\beta
-\sigma}\Vert y\Vert_{\delta}^{\prime}\label{norm_0}\\
\Vert x_{1}\Vert_{\sigma}  &  \leq cN_{1}^{\beta}\Vert y\Vert_{\delta}%
^{\prime}\;\mathrm{and}\;\Vert x_{n+1}-x_{n}\Vert_{\sigma}\leq c\,N_{n}%
^{\kappa\beta}\Vert y\Vert_{\delta}^{\prime} \label{norm_bar}%
\end{align}

\end{theorem}

The set of admissible $\sigma$ and $\beta$ is non-empty. Indeed, because of
(\ref{S}), we have $\frac{\kappa^{2}}{\kappa-1}\mu<S$, so we can find
$\kappa\beta$ and $\sigma$ satisfying condition (\ref{commun}). For
$1<\kappa\leq\left(  1+\sqrt{5}\right)  /2$, we have $\delta>\mu\kappa^{2}$,
so $\kappa\mu-\delta/\kappa<0$ and $\kappa\beta$ can satisfy both
(\ref{commun}) and\ (\ref{beta1}). For $\kappa\geq\left(  1+\sqrt{5}\right)
/2$, we have $\delta-\mu>\kappa\mu$, so $\mu+\sigma-\delta<\sigma-\kappa\mu
\,$\ and condition (\ref{beta2}) is satisfied provided $\beta>\sigma-\kappa
\mu$, or $\kappa\beta>\kappa\sigma-\kappa^{2}\mu$. If $\sigma\ $satisfies
(\ref{commun}), we have $\kappa\sigma-\kappa^{2}\mu>\left(  \kappa-1\right)
\sigma>0$, so we can find $\kappa\beta$ satisfying (\ref{commun}) and
(\ref{beta2}).

Note that the estimate on $\Vert x_{n+1}-x_{n}\Vert_{\sigma}$ blows up very
fast when $n\rightarrow\infty$, while the estimate on $\Vert x_{n+1}%
-x_{n}\Vert_{0}$ goes to zero very fast, since $\kappa\beta-\sigma<0$. Using
the interpolation inequality (\ref{inter}), this will enable us to maintain
control of some intermediate norms.

The proof of Theorem \ref{iteration} is postponed to the next section. We now
show that it implies Theorem \ref{inverse}. Let u begin with an estimate:

\begin{lemma}
\label{estimF}Given $0<A<S$ there is a constant $C_{3}^{A}$ such that, for all
$s\in\lbrack0,A]\,,$ all integers $N,\,P\geq0$ and any $x\in B_{0}(1)\cap
X_{s}$ :%
\begin{align*}
\Vert F(x)\Vert_{s}^{\prime}  &  \leq C_{3}^{A}\Vert x\Vert_{s}\\
\Vert(1-\Pi_{N}^{\prime})F(x)\Vert_{0}^{\prime}  &  \leq C_{3}^{A}N^{-s}\Vert
x\Vert_{s}\\
\Vert\Pi_{N+P}^{\prime}(1-\Pi_{N}^{\prime})F(x)\Vert_{0}^{\prime}  &  \leq
C_{3}^{A}N^{-s}\Vert x\Vert_{s}%
\end{align*}

\end{lemma}

\begin{proof}
The function $\varphi:\,t\in\lbrack0,1]\rightarrow F(tx)$ has derivative
$\frac{d}{dt}\varphi=DF(tx)x$ and by the tame estimates (\ref{tame}) on
$DF\left(  x\right)  $, we have $\Vert\frac{d}{dt}\varphi(t)\Vert_{s}^{\prime
}\leq2K_{A}\Vert x\Vert_{s}$. Since $\varphi(0)=0$, this gives our first
estimate. Combining it with (\ref{gain}), we get the second one, and applying
(\ref{loss}) with $d=0$, we get the third one.
\end{proof}

Let us now prove Theorem \ref{inverse}. Since $\kappa\beta-\sigma<0$, the
inequalities (\ref{norm_0}) imply that the sequence $(x_{n})$ is Cauchy in
$X_{0}$, and has a limit $\bar{x}$ with $\Vert\bar{x}\Vert_{0}\leq C\Vert
\bar{y}\Vert_{\delta}^{\prime}$, where
\[
C=c(N_{1}^{\mu}+\sum_{n\geq1}N_{n}^{\kappa\beta-\sigma})
\]
Then $F(x_{n})$ converges to $F(\bar{x})$ in $Y_{0}\,,$ by the continuity of
$F:X_{s}\rightarrow Y_{s}^{.}$.Similarly, (\ref{norm_bar}) implies that $\Vert
x_{n}\Vert_{\sigma}\leq C_{n}^{\prime}N_{n}^{\kappa\beta}\Vert\bar{y}%
\Vert_{\delta}^{\prime}$, with:
\[
C_{n}^{\prime}:=cN_{n}^{-\kappa\beta}\left(  N_{1}^{\beta}+\sum_{i=1}%
^{n-1}N_{i}^{\kappa\beta}\right)
\]
and $C^{\prime}:=\sup_{n}C_{n}^{\prime}<$ $\infty$, so that $\Vert x_{n}%
\Vert_{\sigma}\leq C^{\prime}N_{n}^{\kappa\beta}\Vert\bar{y}\Vert_{\delta
}^{\prime}$ for all $n$.

\paragraph{The case $1<\kappa\leq\frac{1+\sqrt{5}}{2}$.}

We have, by (\ref{projected}):%
\begin{equation}
F(x_{n})=(1-\Pi_{N_{n}}^{\prime})F(x_{n})+\Pi_{N_{n-1}}^{\prime}\bar{y}\,.
\label{almost}%
\end{equation}

Then Lemma \ref{estimF} gives the estimate
\[
\Vert(1-\Pi_{N_{n}}^{\prime})F(x_{n})\Vert_{0}^{\prime}\leq C_{\delta}%
^{(3)}N_{n}^{-s}\Vert x_{n}\Vert_{s}\text{ \ \ for all \ }s<\delta
\]

Substituting $\Vert x_{n}\Vert_{\sigma}\leq C^{\prime}N_{n}^{\kappa\beta}\Vert
y\Vert_{\delta}^{\prime}$, we get:%
\[
\Vert(1-\Pi_{N_{n}}^{\prime})F(x_{n})\Vert_{0}^{\prime}\leq C^{\prime
}\,C_{\delta}^{(3)}\,N_{n}^{\kappa\beta-\sigma}\Vert y\Vert_{\delta}^{\prime
}\text{ \ \ for all \ }s\leq\delta
\]

By (\ref{commun}), the exponent $(\kappa\beta-\sigma)$ is negative. So
$(1-\Pi_{N_{n}}^{\prime})F(x_{n})$ converges to zero in $Y_{0}$. Now, using
the inequality (\ref{gain}) we get
\[
\Vert(1-\Pi_{N_{n-1}}^{\prime})\bar{y}\Vert_{0}^{\prime}\leq C_{\delta}%
^{(1)}N_{n-1}^{-\delta}\Vert\bar{y}\Vert_{\delta}^{\prime}\text{ \ \ for all
\ }s\leq\delta
\]
so $\Pi_{N_{n-1}}^{\prime}\bar{y}$ converges to $\bar{y}$ in $Y_{0}$. So both
terms on the right-hand side of (\ref{almost}) converge to zero, and we get
$F(\bar{x})=\bar{y}$, as announced.

Together with the interpolation inequality (\ref{inter}), conditions
(\ref{norm_0}) and (\ref{norm_bar}) imply:%
\[
\Vert x_{n+1}-x_{n}\Vert_{\left(  1-t\right)  \sigma}\leq c_{t}N_{n}%
^{\kappa\beta-t\sigma}\left\Vert y\right\Vert _{\delta}^{\prime}%
\]

The exponent on the right-hand side is negative for $t>\kappa\beta/\sigma$, so
that $\left(  1-t\right)  \sigma<\sigma-\kappa\beta$. Arguing as above, if
follows that $\left\Vert \bar{x}\right\Vert _{\alpha}\leq C\left\Vert
y\right\Vert _{\delta}^{\prime}$, provided:%
\[
\alpha<\sup_{\mathcal{A}_{1}}\left\{  \sigma-\kappa\beta\ \right\}
\]
where:%
\[
\mathcal{A}_{1}=\left\{  \left(  \kappa,\beta,\sigma\right)  |\
\begin{array}
[c]{c}%
\sigma-\kappa\beta<\frac{\delta}{\kappa}-\kappa\mu\\
\frac{\kappa^{2}}{\kappa-1}\mu<\kappa\beta<\sigma<S
\end{array}
\right\}
\]

Set $\alpha^{\prime}=\alpha/\mu$, $\sigma^{\prime}=\sigma/\mu$, $\beta
^{\prime}=\beta/\mu$, $\delta.=\delta/\mu,\ S^{\prime}=S/\mu$. The problem
becomes:%
\begin{align*}
\alpha^{\prime}  &  <\sup_{\mathcal{A}_{1}^{\prime}}\left\{  \sigma^{\prime
}-\kappa\beta^{\prime}\ \right\} \\
\mathcal{A}_{1}^{\prime}  &  =\left\{  \left(  \kappa,\beta^{\prime}%
,\sigma^{\prime}\right)  |\
\begin{array}
[c]{c}%
\sigma^{\prime}-\kappa\beta^{\prime}<\frac{\delta^{\prime}}{\kappa}-\kappa\\
\frac{\kappa^{2}}{\kappa-1}<\kappa\beta^{\prime}<\sigma^{\prime}<S^{\prime}%
\end{array}
\right\}
\end{align*}
For given $\kappa$, Figure 3 gives the admissible $\left(  \beta
,\sigma\right)  $ region in the case $S^{\prime}>\delta^{\prime}/\kappa
-\kappa$ (upper horizontal line) and in the case $S^{\prime}<\delta^{\prime
}/\kappa-\kappa$ (lower horizontal line). The admissible region is to the
right of the vertical $\beta^{\prime}=\kappa/\left(  \kappa-1\right)  $, both
in the case $S^{\prime}>\delta^{\prime}/\kappa-\kappa$ (right line) and in the
case $S^{\prime}<\delta^{\prime}/\kappa-\kappa$ (left line)
\begin{center}
\fbox{\includegraphics[
natheight=3.000000in,
natwidth=4.499600in,
height=3in,
width=4.4996in
]%
{C:/Users/ekeland/Dropbox/Sere/graphics/perte-minimale-10__2.png}%
}\\
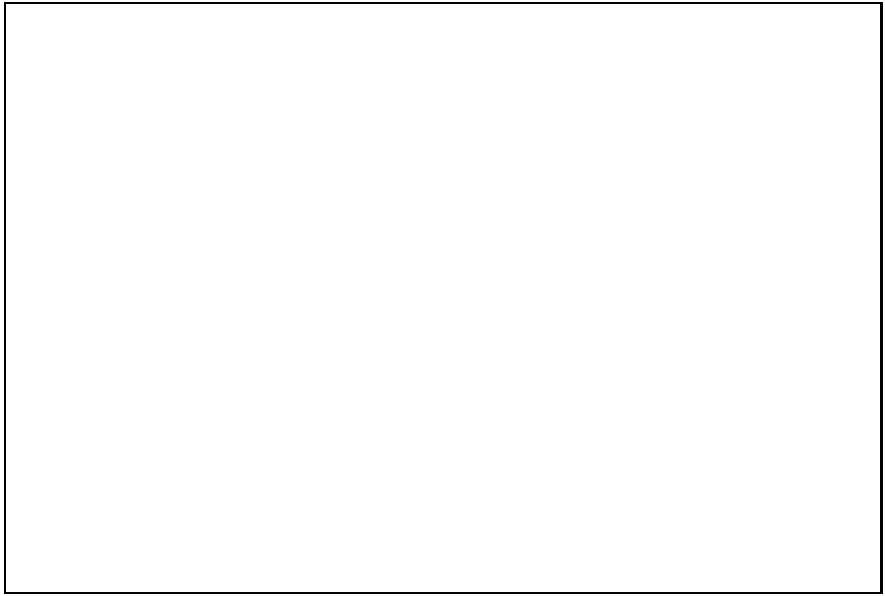
Figure 3:\ The admissible $\left(  \beta,\sigma\right)  $ region in the first
case
\end{center}

The maximum is attained at the upper left corner of the admissible region,
which is the point $\left(  \beta^{\prime},\min\left\{  S^{\prime}%
,\ \kappa\beta^{\prime}-\kappa+\delta^{\prime}/\kappa\right\}  \right)  $,
with $\beta^{\prime}=\kappa\left(  \kappa-1\right)  ^{-1}$. Hence:
\begin{equation}
\sup_{\mathcal{A}_{1}}\left\{  \sigma^{\prime}-\kappa\beta^{\prime}\right\}
=\min\left\{  S^{\prime}-\frac{\kappa^{2}}{\kappa-1},\ \frac{\delta^{\prime}%
}{\kappa}-\kappa\right\}  \label{n1}%
\end{equation}

\paragraph{The case $\frac{1+\sqrt{5}}{2}\leq\kappa<2$}

The argument is the same, except that we have to replace $\Pi_{N_{n-1}%
}^{\prime}\bar{y}$ by $\Pi_{N_{n}}^{\prime}\bar{y}$ in (\ref{almost}).

We now have:%
\begin{align*}
\alpha^{\prime}  &  <\sup_{\mathcal{A}_{2}^{\prime}}\left\{  \sigma^{\prime
}-\kappa\beta^{\prime}\ \right\} \\
\mathcal{A}_{2}^{\prime}  &  =\left\{  \left(  \kappa,\beta^{\prime}%
,\sigma^{\prime}\right)  |\
\begin{array}
[c]{c}%
\sigma^{\prime}<\beta^{\prime}+\delta-1\\
\frac{\kappa^{2}}{\kappa-1}<\kappa\beta^{\prime}<\sigma^{\prime}<S^{\prime}%
\end{array}
\right\}
\end{align*}

For given $\kappa$, the admissible $\left(  \beta,\sigma\right)  $ region is
given in Figure 4, in the case $S>\delta-1+\kappa/\left(  \kappa-1\right)  $
(upper horizontal line) and in the case $S<\delta-1+\kappa/\left(
\kappa-1\right)  $ (lower horizontal line). The vertical is $\beta
=\kappa/\left(  \kappa-1\right)  $.%
\begin{center}
\fbox{\includegraphics[
natheight=3.000000in,
natwidth=4.499600in,
height=3in,
width=4.4996in
]%
{C:/Users/ekeland/Dropbox/Sere/graphics/perte-minimale-10__3.png}%
}\\
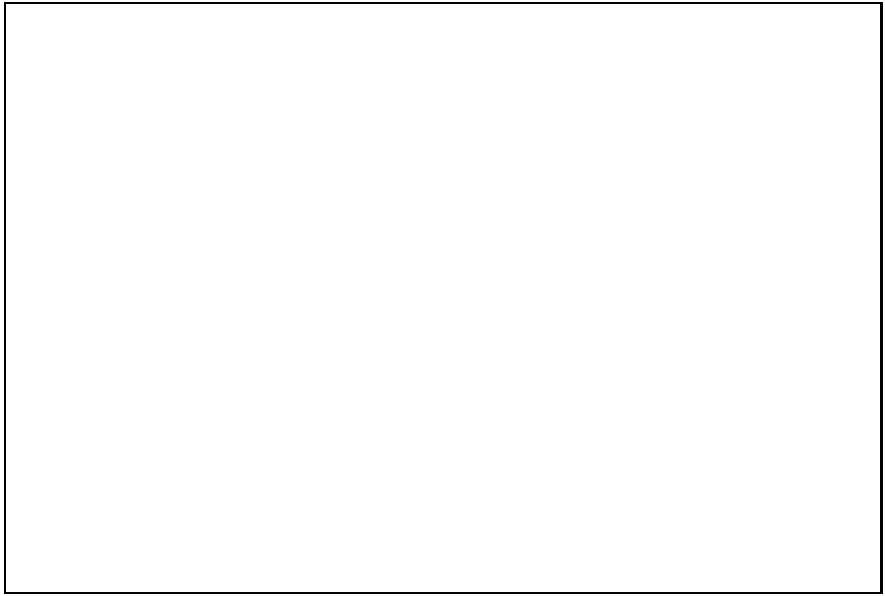
Figure 4:\ The admissible $\left(  \beta,\sigma\right)  $ region in the second
case
\end{center}

Again the maximum is attained in the upper left corner, which is the point
$\left(  \beta^{\prime},\ \min\left\{  S^{\prime},\ \delta^{\prime}%
-1+\beta^{\prime}\right\}  \right)  $ with $\beta^{\prime}=\kappa\left(
\kappa-1\right)  $. This gives:%
\begin{equation}
\sup_{\mathcal{A}_{2}}\left\{  \sigma^{\prime}-\kappa\beta^{\prime}\ \right\}
=\min\left\{  S^{\prime}-\frac{\kappa^{2}}{\kappa-1},\ \delta^{\prime
}-1-\kappa\right\}  \label{n2}%
\end{equation}

Putting (\ref{n1}) and (\ref{n2}) together gives formula (\ref{n3})

\section{Proof of Theorem 2\label{2}}

We work under the assumptions of Theorem \ref{iteration}. So $\mu
,\delta,S,\kappa,\sigma,\beta,\bar{y}$ are given. Note that we may have
$\sigma<\delta.$

We assume $\bar{y}\neq0$ (the case $\bar{y}=0$ is obvious). We fix
$A=\sigma\,,$ and the constants $C_{1}^{A},$ $C_{2}^{A},\ C_{3}^{A},\ K^{A},$
$\gamma^{A}$ of (\ref{loss}, \ref{gain}, \ref{inter}, \ref{tame},
\ref{tameinverse}) and Lemma \ref{estimF} are simply denoted $C_{1},$
$C_{2}^{A},\ C_{3}^{A},\ K\,,$ $\gamma$. The proof will make use of a certain
number of constants, which we list here to make sure that they do not depend
on the iteration step and can be fixed at the beginning.

Recall that $2^{\left(  \alpha\right)  }$ is the integer part of $2^{\alpha}$,
and set $P_{n}=E\left[  2^{\kappa^{n}}\right]  $. There is a constant $g>1$
such that for all $N_{0}\geq2$ and $n\geq0$,%
\begin{equation}
g^{-1}P_{n}^{\kappa}\leq P_{n+1}\leq g\,P_{n}^{\kappa} \label{equiv}%
\end{equation}

We define constants $B_{0},\ B_{1}$ and $B_{2}$ by:%
\begin{align}
B_{0}  &  =(P_{1}+\sum_{n\geq1}P_{n}^{\kappa\beta-\sigma})^{-1}\label{h1}\\
B_{1}  &  :=\sup_{n}\{P_{n}^{-\beta}(P_{1}^{\beta}+\sum_{1\leq i\leq n-1}%
P_{i}^{\kappa\beta})\;|\;n\geq1\}\label{h2}\\
B^{2}  &  :=\sup_{n}\left\{  B_{1}P_{n}^{-\left(  \kappa-1\right)  \beta
}+1\ |\ n\geq1\right\}  \label{h3}%
\end{align}

We shall use Theorem \ref{Ivar} to construct inductively the sequence $x_{n}$,
thanks to a sequence of carefully chosen norms. For this purpose, we will have
to take $c$ large and $\rho$ small.

\section{Choice of $N_{0}$}

For $1<\kappa\leq\frac{1+\sqrt{5}}{2}$, we consider the following function of
$n$:
\begin{equation}
\varphi_{2}\left(  n\right)  :=2\,B_{1}C_{3}\,c\,P_{n}^{\beta-\kappa\left(
\beta-\mu\right)  }+C_{1}\left(  P_{n}^{\sigma-\delta/\kappa-\kappa\left(
\beta-\mu\right)  }g^{\delta/\kappa}+P_{n}^{(\sigma-\delta)_{+}-(\delta
-\sigma)_{+}/\kappa-\kappa\left(  \beta-\mu\right)  }\right)  \label{h4}%
\end{equation}

By condition (\ref{commun}) and (\ref{beta1}), all the exponents are negative.
So we may pick $n_{0}$ so large that:%
\begin{equation}
\varphi_{2}\left(  n_{0}\right)  \leq\gamma c\left(  B_{2}+2\right)
^{-1}g^{-\mu}\text{ for all }n\geq n_{0} \label{h5}%
\end{equation}

For$\frac{1+\sqrt{5}}{2}\leq\kappa<2$, we consider the following function of
$N_{0}$ and $n$:
\begin{equation}
\varphi_{1}\left(  n_{0}\right)  :=2\,B_{1}C_{3}\,c\,P_{n}^{\beta-\kappa
(\beta-\mu)}+C_{1}\left(  P_{n}^{\sigma-\delta-\kappa(\beta-\mu)}%
+g^{(\sigma-\delta)_{+}}P_{n}^{\kappa(\sigma-\delta)_{+}-(\delta-\sigma
)_{+}-\kappa(\beta-\mu)}\right)  \label{h9}%
\end{equation}

By condition (\ref{commun}) and (\ref{beta2}), all the exponents are negative.
So we may pick $N_{0}$ so large that:%
\begin{equation}
\varphi_{1}\left(  n_{0}\right)  \leq\gamma c\left(  B_{2}+2\right)
^{-1}g^{-\mu}\text{ for all }n\geq n_{0} \label{h8}%
\end{equation}

In both cases we set $N_{0}=P_{n_{0}}$, and $N_{n}=E\left[  N_{0}^{n^{\kappa}%
}\right]  =E\left[  2^{\left(  n_{o}n\right)  ^{\kappa}}\right]  $. So the
expressions (\ref{h4}) and (\ref{h5}) are less than $\gamma c\left(
B_{2}+2\right)  ^{-1}g^{-\mu}$ when one substitutes $N_{n}$ for $P_{n}$.

\subsection{Construction of the initial point}

\paragraph{The case $1<\kappa\leq\frac{1+\sqrt{5}}{2}$.}

Thanks to inequality (\ref{gain}), $\Vert\Pi_{N_{0}}^{\prime}\bar{y}\Vert
_{0}\leq C_{1}\Vert\bar{y}\Vert_{\delta}^{\prime}$ and $\Vert\Pi_{N_{0}%
}^{\prime}\bar{y}\Vert_{\sigma}\leq C_{1}N_{0}^{(\sigma-\delta)_{+}}\Vert
\bar{y}\Vert_{\delta}^{\prime}$, where $t_{+}$ denotes the positive part of
the real number $t$. We choose the norm
\[
\mathcal{N}_{0}(x)=\Vert x\Vert_{0}+N_{0}^{-(\sigma-\delta)_{+}}\Vert
x\Vert_{\sigma}%
\]
on $\,E_{N_{1}}$ and the norm:
\[
\mathcal{N}_{0}^{\prime}(y)=\Vert y\Vert_{0}^{\prime}+N_{0}^{-(\sigma
-\delta)_{+}}\Vert y\Vert_{\sigma}^{\prime}%
\]
on $\,E_{N_{1}}^{\prime}$. For these norms, $E_{N_{1}}$ and $E_{N_{1}}%
^{\prime}$ are Banach spaces. Note that%
\begin{equation}
\mathcal{N}_{0}^{\prime}(\Pi_{N_{0}}^{\prime}y)<2C_{1}\Vert y\Vert_{\delta
}\text{ \ for }y\in E_{N_{1}}^{\prime} \label{b1}%
\end{equation}

For $\mathcal{N}_{0}\left(  x\right)  \leq1$, we define
\[
f(x):=\Pi_{N_{1}}^{\prime}F(x)\in E_{N_{1}}^{\prime}%
\]

The function $f$ is continuous and G\^{a}teaux-differentiable for the norms
$\mathcal{N}_{0}$ and $\mathcal{N}_{0}^{\prime}$, with $f\left(  0\right)
=0$. Moreover, using the tame estimate (\ref{tameinverse}) and applying
assumption (\ref{loss}) to $\left\Vert x\right\Vert _{\sigma}$, we find that:%
\begin{align}
\sup\left\{  \left\Vert \,[Df(x)]^{-1}k\right\Vert _{0}\ |\ \mathcal{N}\left(
x\right)  \leq1\right\}   &  <\frac{2N_{1}^{\mu}}{\gamma}\Vert k\Vert
_{0}^{\prime}\label{d1}\\
\sup\left\{  \left\Vert \,[Df(x)]^{-1}k\right\Vert _{\sigma}\ |\ \mathcal{N}%
\left(  x\right)  \leq1\right\}   &  <\frac{N_{1}^{\mu}}{\gamma}(\Vert
k\Vert_{\sigma}^{\prime}+N_{0}^{(\sigma-\delta)_{+}}\Vert k\Vert_{0}^{\prime})
\label{d2}%
\end{align}
hence:
\begin{equation}
\sup\left\{  \mathcal{N}_{0}([Df(x)]^{-1}k)\ |\ \mathcal{N}\left(  x\right)
\leq1\right\}  <\frac{3N_{1}^{\mu}}{\gamma}\mathcal{N}_{0}^{\prime}(k)
\label{d3}%
\end{equation}

By Theorem \ref{Ivar}, we can solve $f\left(  \bar{u}\right)  =\bar{v}$ with
$\mathcal{N}_{0}(\bar{u})\leq1$ if $\mathcal{N}_{0}^{\prime}(\Pi_{N_{0}%
}^{\prime}\bar{y})\leq\gamma\left(  3N_{1}^{\mu}\right)  ^{-1}$. By
(\ref{b1}), this is fulfilled provided:%
\begin{equation}
\left\Vert \bar{y}\right\Vert _{\delta}^{\prime}\leq\frac{\gamma}{6C_{1}%
N_{1}^{\mu}}=:\rho\;. \label{rho1}%
\end{equation}

In addition, Theorem \ref{Ivar} tells us that we have the estimate:%
\begin{equation}
\mathcal{N}_{0}(\bar{u})\leq3N_{1}^{\mu}\gamma^{-1}\mathcal{N}_{0}^{\prime
}\left(  \Pi_{N_{0}}^{\prime}\bar{y}\right)  \leq6C^{(1)}N_{1}^{\mu}%
\gamma^{-1}\left\Vert \bar{y}\right\Vert _{\delta}^{\prime} \label{d4}%
\end{equation}

If (\ref{rho1} is satisfied, $x_{1}:=\bar{u}$ is the desired solution in
$E_{N_{1}}$ of the projected equation $\Pi_{N_{1}}^{\prime}F(x_{1})=\Pi
_{N_{0}}^{\prime}\bar{y}$, with $\mathcal{N}_{0}(\bar{u})\leq1$. Let us check
conditions (\ref{norm_bar}) and (\ref{norm_0}). We have, by $\left(
\ref{d4}\right)  $:%
\[
\mathcal{N}_{0}(x_{1})=\Vert x_{1}\Vert_{0}+N_{0}^{-(\sigma-\delta)_{+}}\Vert
x_{1}\Vert_{\sigma}\leq R=6C^{(1)}N_{1}^{\mu}\gamma^{-1}\Vert\bar{y}%
\Vert_{\delta}^{\prime}%
\]

Since $\ N_{0}\leq g^{1/\kappa}N_{1}^{1/\kappa}$, we find:%
\[
\Vert x_{1}\Vert_{0}+\left(  gN_{1}\right)  ^{-\kappa^{-1}(\sigma-\delta)_{+}%
}\Vert x_{1}\Vert_{\sigma}\leq6C^{(1)}N_{1}^{\mu}\gamma^{-1}\Vert\bar{y}%
\Vert_{\delta}^{\prime}%
\]
\ 

Since $\mu+\kappa^{-1}(\sigma-\delta)_{+}<\beta$, this yields $\Vert
x_{1}\Vert_{0}\leq cN_{1}^{\mu}\Vert\bar{y}\Vert_{\delta}^{\prime}$ and $\Vert
x_{1}\Vert_{\sigma}\leq cN_{1}^{\beta}\Vert\bar{y}\Vert_{\delta}^{\prime}$ as
required, with
\begin{equation}
c:=6C^{(1)}g^{(\sigma-\delta)_{+}/\kappa}\gamma^{-1} \label{c1c}%
\end{equation}

\paragraph{The case $\frac{1+\sqrt{5}}{2}\leq\kappa<2$}

Very few modifications are needed in the above arguments. Replace $N_{0}$ by
$N_{1}$, so that the norms become:%
\begin{align*}
\mathcal{N}_{0}(x)  &  =\Vert x\Vert_{0}+N_{1}^{-(\sigma-\delta)_{+}}\Vert
x\Vert_{\sigma}\\
\mathcal{N}_{0}^{\prime}(y)  &  =\Vert y\Vert_{0}^{\prime}+N_{1}%
^{-(\sigma-\delta)_{+}}\Vert y\Vert_{\sigma}^{\prime}%
\end{align*}
and define as above $f(x):=\Pi_{N_{1}}^{\prime}F(x)\in E_{N_{1}}^{\prime}$.
Because (\ref{projected}) is replaced by (\ref{projected'}), we now consider
$\Pi_{N_{1}}^{\prime}\bar{y}\in E_{N_{1}}^{\prime}$. Estimates (\ref{b1}) and
(\ref{d3}) still hold. Using Theorem \ref{Ivar} as before, we will be able do
find some $\bar{u}\in E_{N_{1}}$ with $\mathcal{N}_{0}(\bar{u})\leq1$ solving
$f\left(  \bar{u}\right)  =\Pi_{N_{1}}^{\prime}\bar{y}$ provided $\bar{y}$
satisfies (\ref{rho1}). The estimate (\ref{d4}) still holds:%
\[
\mathcal{N}_{0}(x_{1})=\Vert x_{1}\Vert_{0}+N_{1}^{-(\sigma-\delta)_{+}}\Vert
x_{1}\Vert_{\sigma}\leq6C_{1}N_{1}^{\mu}\gamma^{-1}\Vert\bar{y}\Vert_{\delta
}^{\prime}%
\]

Since $\mu+(\sigma-\delta)_{+}<\beta$, this yields $\Vert x_{1}\Vert_{0}\leq
cN_{1}^{\mu}\Vert\bar{y}\Vert_{\delta}^{\prime}$ and $\Vert x_{1}\Vert
_{\sigma}\leq cN_{1}^{\beta}\Vert\bar{y}\Vert_{\delta}^{\prime}$ as above,
with
\begin{equation}
c:=6C^{(1)}\gamma^{-1} \label{c2c}%
\end{equation}

Diminishing $\rho$, if necessary, we can always assume that, in both cases,
$\rho$ and $c$ satisfy the constraint, where B$_{0}$ is defined by (\ref{h1}):%
\begin{equation}
c\rho<B^{\left(  0\right)  } \label{bound}%
\end{equation}

\subsection{Induction}

\paragraph{The case $1<\kappa\leq\frac{1+\sqrt{5}}{2}$.}

Assume that the result has been proved up to $n$. In other words, define $c$
by (\ref{c1c}), and assume we have found $\rho$ with $c\rho<B^{\left(
0\right)  }$ such that, for $\bar{y}\in Y$ with $\left\Vert \bar{y}\right\Vert
_{\delta}^{\prime}\leq\rho$, there exists a sequence $x_{1},\cdots,x_{n}$
satisfying (\ref{projected}), (\ref{norm_bar}), and (\ref{norm_0}). To be
precise:%
\begin{align}
\Vert x_{i+1}-x_{i}\Vert_{0}  &  \leq c\,N_{i}^{\kappa\beta-\sigma}\Vert
\bar{y}\Vert_{\delta}^{\prime}\text{ \ for }i\leq n-1\label{d5}\\
\Vert x_{i+1}-x_{i}\Vert_{\sigma}  &  \leq c\,N_{i}^{\kappa\beta}\Vert\bar
{y}\Vert_{\delta}^{\prime}\text{\ \ for }i\leq n-1 \label{d6}%
\end{align}

Since $x_{1},\cdots,x_{n}$ satisfy (\ref{norm_0}), and $\left\Vert \bar
{y}\right\Vert _{\delta}^{\prime}\leq\rho$, this will imply that $\Vert
x_{n}\Vert_{0}\leq1-\eta_{n}$ with:%

\begin{equation}
\eta_{n}:=\frac{\sum_{i\geq n}N_{i}^{\kappa\beta-\sigma}}{N_{1}^{\mu}%
+\sum_{n\geq1}N_{n}^{\kappa\beta-\sigma}}\leq c\rho\sum_{i\geq n}N_{i}%
^{\kappa\beta-\sigma}\,. \label{e1}%
\end{equation}

Since $x_{1},\cdots,x_{n}$ satisfy (\ref{norm_bar}), we also have:
\[
\Vert x_{n}\Vert_{\sigma}\leq c\,(N_{1}^{\beta}+\sum_{1\leq i\leq n-1}%
N_{i}^{\kappa\beta})\Vert\bar{y}\Vert_{\delta}^{\prime}%
\]

Using the constant $B_{1}$ defined in (\ref{h1}), this becomes:%
\begin{equation}
\Vert x_{n}\Vert_{\sigma}\leq B_{1}\,c\,N_{n}^{\beta}\Vert\bar{y}\Vert
_{\delta} \label{f1}%
\end{equation}

We are going to construct $x_{n+1}$ so that (\ref{projected}), (\ref{norm_bar}%
), and (\ref{norm_0}) hold for $i\leq n$. Write:%
\[
x_{n+1}=x_{n}+\Delta x_{n}%
\]

By the induction hypothesis, $x_{n}\in E_{N_{n}}$ and $\Pi_{N_{n}}^{\prime
}F(x_{n})=\Pi_{N_{n-1}}^{\prime}\bar{y}$. The equation to be solved by $\Delta
x_{n}$ may be written in the following form:%
\begin{align}
&  f_{n}\left(  \Delta x_{n}\right)  =e_{n}+\Delta y_{n-1}\label{d10}\\
f_{n}(u)  &  :=\Pi_{N_{n+1}}\left(  F(x_{n}+u)-F(x_{n})\right) \label{d11}\\
e_{n}  &  :=\Pi_{N_{n+1}}(\Pi_{N_{n}}-1)F(x_{n})\label{d12}\\
\Delta y_{n-1}  &  :=\Pi_{N_{n}}(1-\Pi_{N_{n-1}})\bar{y} \label{d13}%
\end{align}

The function $f_{n}$ is continuous and G\^{a}teaux-differentiable with
$f\left(  0\right)  =0$. We will solve equation (\ref{d10}) by applying
Theorem \ref{Ivar}. We choose the norms:
\begin{align}
\mathcal{N}_{n}(u)  &  =\Vert u\Vert_{0}+N_{n}^{-\sigma}\Vert u\Vert_{\sigma
}\text{ \ on \ }\,E_{N_{n+1}}\label{f12}\\
\mathcal{N}_{n}^{\prime}(v)  &  =\Vert v\Vert_{0}^{\prime}+N_{n}^{-\sigma
}\Vert v\Vert_{\sigma}^{\prime}\text{ \ on \ }\,E_{N_{n+1}}^{\prime}
\label{f13}%
\end{align}
$\,$

Let $R_{n}:=cN_{n}^{\kappa\beta-\sigma}\Vert\bar{y}\Vert_{\delta}\,$. If
$\mathcal{N}_{n}(u)\leq R_{n}$, we have, by (\ref{e1})
\[
\Vert u\Vert_{0}\leq\mathcal{N}_{n}(u)<R_{n}<cN_{n}^{\kappa\beta-\sigma}%
\rho<\eta_{n}%
\]
so that $\left\Vert x_{n}+u\right\Vert _{0}\leq1$ and the function $f_{n}$ is
well-defined by (\ref{d11}). Using (\ref{f1}) and\ (\ref{f12}), we find that,
for $\mathcal{N}_{n}(u)\leq R_{n}$, we have $\left\Vert u\right\Vert _{\sigma
}\leq N_{n}^{\sigma}R_{n}=cN_{n}^{\kappa\beta}\Vert\bar{y}\Vert_{\delta}$, and
hence, with the constants $B_{1}$ and $B_{2}$defined by (\ref{h2}) and
(\ref{h3}):
\begin{equation}
\Vert x_{n}+u\Vert_{\sigma}\leq c\,(B_{1}\,N_{n}^{\beta}+N_{n}^{\kappa\beta
})\Vert\bar{y}\Vert_{\delta}\leq B_{2}\,c\,N_{n}^{\kappa\beta}\Vert\bar
{y}\Vert_{\delta}^{\prime} \label{f15}%
\end{equation}

Plugging this into the tame inequality (\ref{tameinverse}), and taking into
account that $c\Vert\bar{y}\Vert_{\delta}=R_{n}N_{n}^{\sigma-\kappa\beta}$
then gives:%
\begin{align}
\sup\left\{  \left\Vert \,[Df_{n}(u)]^{-1}k\right\Vert _{0}\ |\ \mathcal{N}%
_{n}(u)\leq R_{n}\right\}   &  <2N_{n+1}^{\mu}\gamma^{-1}\Vert k\Vert
_{0}^{\prime}\label{f20}\\
\sup\left\{  \left\Vert \,[Df_{n}(u)]^{-1}k\right\Vert _{\sigma}%
\ |\ \mathcal{N}_{n}(u)\leq R_{n}\right\}   &  <N_{n+1}^{\mu}\gamma^{-1}(\Vert
k\Vert_{\sigma}^{\prime}+B_{2}cN_{n}^{\kappa\beta}\Vert k\Vert_{0}^{\prime
}\Vert\bar{y}\Vert_{\delta}^{\prime})\nonumber\\
&  =N_{n+1}^{\mu}\gamma^{-1}(\Vert k\Vert_{\sigma}^{\prime}+B_{2}%
\,R_{n}\,N_{n}^{\sigma}\Vert k\Vert_{0}^{\prime}) \label{f21}%
\end{align}

Hence:
\[
\sup\left\{  \mathcal{N}_{n}([Df_{n}(u)]^{-1}k)\ |\ \mathcal{N}_{n}(u)\leq
R_{n}\right\}  <\gamma^{-1}\left(  B^{\left(  2\right)  }+2\right)
N_{n+1}^{\mu}\mathcal{N}_{n}^{\prime}(k)
\]

By Theorem \ref{Ivar}, we will be able to solve (\ref{d10}) with
$\mathcal{N}_{n}(\Delta x_{n})\leq R_{n}$ provided:%
\begin{equation}
\mathcal{N}_{n}^{\prime}\left(  e_{n}+\Delta y_{n-1}\right)  \leq\gamma\left(
B^{\left(  2\right)  }+2\right)  ^{-1}N_{n+1}^{-\mu}R_{n}=\gamma c\left(
B^{\left(  2\right)  }+2\right)  ^{-1}N_{n+1}^{-\mu}N_{n}^{\kappa\beta-\sigma
}\Vert\bar{y}\Vert_{\delta} \label{d25}%
\end{equation}

We now compute $\mathcal{N}_{n}^{\prime}\left(  e_{n}+\Delta y_{n-1}\right)
$. Applying inequality (\ref{gain}) to (\ref{d13}), we get:%
\begin{align}
\Vert\Delta y_{n-1}\Vert_{0}^{\prime}  &  \leq C_{1}N_{n-1}^{-\delta}\Vert
\bar{y}\Vert_{\delta}^{\prime}\label{Dy}\\
\Vert\Delta y_{n-1}\Vert_{\sigma}^{\prime}  &  \leq C_{1}N_{n-1}%
^{-(\delta-\sigma)_{+}}N_{n}^{(\sigma-\delta)_{+}}\Vert\bar{y}\Vert_{\delta
}^{\prime} \label{dy2}%
\end{align}

Because of the induction hypothesis, we get
\begin{equation}
\Vert x_{n}\Vert_{\sigma}\leq c\,(N_{1}^{\beta}+\sum_{1\leq i\leq n-1}%
N_{i}^{\kappa\beta})\Vert\bar{y}\Vert_{\delta}^{\prime}\leq B_{1}%
\,c\,N_{n}^{\beta}\Vert\bar{y}\Vert_{\delta}^{\prime} \label{e2}%
\end{equation}

So, combining Lemma \ref{estimF} with conditions (\ref{norm_bar}) and
(\ref{norm_0}), we find that%
\begin{align}
\Vert e_{n}\Vert_{0}^{\prime}  &  \leq C^{(3)}N^{-\sigma}\Vert x_{n}%
\Vert_{\sigma}\leq B_{1}C_{3}\,c\,N_{n}^{\beta-\sigma}\Vert\bar{y}%
\Vert_{\delta}^{\prime}\label{residu}\\
\Vert e_{n}\Vert_{\sigma}^{\prime}  &  \leq C^{(3)}\Vert x_{n}\Vert_{\sigma
}\leq B_{1}C_{3}\,c\,N_{n}^{\beta}\Vert\bar{y}\Vert_{\delta}^{\prime}
\label{res2}%
\end{align}

We now check (\ref{d25}). Using the estimates (\ref{Dy}), (\ref{dy2}),
(\thinspace\ref{residu}), (\ref{res2}), and remembering (\ref{equiv}), we have :%

\begin{align}
\mathcal{N}_{n}^{\prime}(e_{n})  &  \leq2B_{1}C_{3}\,c\,N_{n}^{\beta-\sigma
}\Vert\bar{y}\Vert_{\delta}^{\prime}\label{g1}\\
\mathcal{N}_{n}^{\prime}(\Delta y_{n-1})  &  \leq C_{1}\Vert\bar{y}%
\Vert_{\delta}^{\prime}\left(  N_{n-1}^{-\delta}+N_{n-1}^{-\sigma}%
N_{n-1}^{-(\delta-\sigma)_{+}}N_{n}^{(\sigma-\delta)_{+}}\right) \label{g2}\\
\mathcal{N}_{n}^{\prime}(e_{n})+\mathcal{N}_{n}^{\prime}(\Delta y_{n-1})  &
\leq\left[  2\,B_{1}C_{3}\,c\,N_{n}^{\beta-\sigma}+C_{1}\left(  N_{n-1}%
^{-\delta}+N_{n-1}^{-\sigma-(\delta-\sigma)_{+}}N_{n}^{(\sigma-\delta)_{+}%
}\right)  \right]  \left\Vert \bar{y}\right\Vert _{\delta}^{\prime} \label{g3}%
\end{align}

Condition (\ref{d25}) is satisfied provided:%
\begin{equation}
2\,B_{1}C_{3}\,c\,N_{n}^{\beta}+C_{1}\left(  N_{n}^{\sigma-\delta/\kappa
}g^{\delta/\kappa}+N_{n}^{(\sigma-\delta)_{+}-(\delta-\sigma)_{+}/\kappa
}\right)  \leq\gamma c\left(  B_{2}+2\right)  ^{-1}g^{-\mu}N_{n}%
^{\kappa\left(  \beta-\mu\right)  } \label{h6}%
\end{equation}

Dividing by $N_{n}^{\kappa\left(  \beta-\mu\right)  }$ on both sides, we find
that the sequence on the left-hand side is just a subsequence of $\varphi
_{2}\left(  n\right)  $, where $\varphi_{2}$ is defined by (\ref{h4}), and so
(\ref{h6}) follows directly from (\ref{h5}), that is, from the construction of
$N_{0}$. So we may apply Theorem \ref{Ivar}, and we find $u=\Delta x_{n}$ with
$\mathcal{N}_{n}(\Delta x_{n})\leq R_{n}$ solving (\ref{d10}). Since
$\mathcal{N}_{n}(\Delta x_{n})\leq R_{n}$, we have $\Vert\Delta x_{n}%
\Vert_{\sigma}\leq N_{n}^{\sigma}R_{n}=cN_{n}^{\kappa\beta}\Vert\bar{y}%
\Vert_{\delta}^{\prime}$ and $\Vert\Delta x_{n}\Vert_{0}\leq R_{n}%
=cN_{n}^{\kappa\beta-\sigma}\Vert\bar{y}\Vert_{\delta}^{\prime}$ , so
inequalities (\ref{norm_0}) and (\ref{norm_bar}) are satisfied. The induction
is proved.

\paragraph{The case $\frac{1+\sqrt{5}}{2}\leq\kappa<2.$}

The induction hypothesis becomes $\Pi_{N_{n}}^{\prime}F(x_{n})=\Pi_{N_{n}%
}^{\prime}\bar{y}$. The system (\ref{d10}) (\ref{d11}), (\ref{d12}),
\ref{d13}) becomes:%
\begin{align*}
&  f_{n}\left(  \Delta x_{n}\right)  =e_{n}+\Delta y_{n}\\
f_{n}(u) &  :=\Pi_{N_{n+1}}\left(  F(x_{n}+u)-F(x_{n})\right)  \\
e_{n} &  :=\Pi_{N_{n+1}}(\Pi_{N_{n}}-1)F(x_{n})\\
\Delta y_{n} &  :=\Pi_{N_{n}}(1-\Pi_{N_{n}})\bar{y}%
\end{align*}

Using Theorem \ref{Ivar} in the same way, we find that we can find $\Delta
x_{n}$ satisfying these equations and the estimates (\ref{norm_0}) and
(\ref{norm_bar}) provided:
\begin{equation}
2\,B^{\left(  3\right)  }\,c\,N_{n}^{\beta}+C_{1}\left(  N_{n}^{\sigma-\delta
}+g^{(\sigma-\delta)_{+}}N_{n}^{\kappa(\sigma-\delta)_{+}-(\delta-\sigma)_{+}%
}\right)  \leq\gamma c\left(  B^{\left(  2\right)  }+2\right)  ^{-1}g^{-\mu
}N_{n}^{\kappa\left(  \beta-\mu\right)  } \label{condition2'}%
\end{equation}

But the left-hand side is just $\varphi_{1}\left(  n,N_{0}\right)
N_{n}^{\kappa\left(  \beta-\mu\right)  }$, where $\varphi_{1}$ is defined by
(\ref{h9}), and so (\ref{condition2'}) follows directly from (\ref{h8}), that
is, from the construction of $N_{0}$. The induction is proved in this case as well.

\subsection{Proof of Proposition \ref{Prop}}

We will take advantage of the special form of $e_{n}:=\Pi_{N_{n+1}}(\Pi
_{N_{n}}-1)F(x_{n})$.The proof is the same, with the estimates (\ref{residu})
and (\ref{res2}) derived as follows. Since $x_{n}\in E_{N_{n}}$, and
$E_{N_{n}}$ is $A$-invariant, we have $\Pi_{N_{n}}Ax_{n}=Ax_{n}$ and:
\begin{align*}
e_{n}  &  :=\Pi_{N_{n+1}}(\Pi_{N_{n}}-1)F(x_{n})\\
&  =\Pi_{N_{n+1}}(\Pi_{N_{n}}-1)\left(  Ax_{n}+G(x_{n})\right) \\
&  =\Pi_{N_{n+1}}(\Pi_{N_{n}}-1)G(x_{n})
\end{align*}

Lemma \ref{estimF} holds for $G$ (though no longer for $F$), so that estimates
(\ref{residu}) and (\ref{res2}) follow readily.

\section{ An isometric imbedding\label{3}}

We will use the same example as Moser in his seminal paper \cite{Moser}, who
himself follows Nash \cite{Nash}. Suppose we\ are given a Riemannian structure
$g^{0}$ on the two-dimensional torus $\mathbb{T}_{2}=\left(  \mathbb{R}%
/\mathbb{Z}\right)  ^{2}$, and an isometric imbedding into Euclidian
$\mathbb{R}$$^{5}$. In other words, we know $x^{0}=\left(  x_{1}^{0}%
,...,x_{5}^{0}\right)  $ with:%
\[
\left(  \frac{\partial x^{0}}{\partial\theta_{i}},\frac{\partial x^{0}%
}{\partial\theta_{j}}\right)  =g_{i,j}^{0}\left(  \theta_{1},\theta
_{2}\right)
\]
where $g_{i,j}^{0}=g_{j,i}^{0}$, so there are three equations for five unknown
functions. If we slightly perturb the Riemannian structure, does the imbedding
still exist ?\ If we replace $g^{0}$ on the right-hand side by some $g$
sufficiently close to $g^{0}$, can we still find some $x:\mathbb{T}%
_{2}\rightarrow\mathbb{R}$$^{5}$ which solves the system ?

We consider the Sobolev spaces $H^{s}\left(  \mathbb{T}_{2};\mathbb{R}%
^{5}\right)  $ and we assume that $x^{0}\in H^{\mu}$, with $\mu>3$, and
$g^{0}\in H^{\sigma}$. Define $F=\left(  F^{i,j}\right)  $ by:%

\begin{equation}
F_{i,j}\left(  x+x^{0}\right)  =\left(  \frac{\partial}{\partial\theta_{i}%
}\left(  x+x^{0}\right)  ,\frac{\partial x}{\partial\theta_{j}}\left(
x+x^{0}\right)  \right)  -g^{0} \label{h20}%
\end{equation}

Clearly $F\left(  0\right)  =0$. For $s\geq3/2$, we know $H^{s}$ is an
algebra, so if $s\geq\mu$ and $x\in H^{s}$, the first term on the right is in
$H^{s-1}$. On the other hand, the right-hand side cannot be more regular than
$g^{0}$, which is in $H^{\sigma}$. So $F$ sends $H^{s}$ into $H^{s-1}$ for
$\mu\leq s\leq\sigma+1$. The function $F$ is quadratic, hence smooth, and we
have:%
\begin{equation}
\left[  DF\left(  x\right)  u\right]  _{i,j}=\left(  \frac{\partial
x}{\partial\theta_{i}},\frac{\partial u}{\partial\theta_{j}}\right)  +\left(
\frac{\partial u}{\partial\theta_{i}},\frac{\partial x}{\partial\theta_{j}%
}\right)  \label{h10}%
\end{equation}

We need to invert the derivative $DF\left(  x\right)  $, that is, to solve the
system%
\begin{equation}
DF\left(  x\right)  u=v_{i,j} \label{h14}%
\end{equation}

It is an undetermined system, since there are three equations for five
unknowns. Following Nash, and Moser, we impose two additional conditions:%
\begin{equation}
\left(  \frac{\partial x}{\partial\theta_{1}},u\right)  =\left(
\frac{\partial x}{\partial\theta_{2}},u\right)  =0 \label{h11}%
\end{equation}

Differentiating, and substituting into (\ref{h10}), we find:%
\begin{equation}
-2\left(  \frac{\partial^{2}x}{\partial\theta_{i}\partial\theta_{j}},u\right)
=v_{i,j}\text{ } \label{h12}%
\end{equation}

So any solution $\varphi$ of (\ref{h11}), (\ref{h12}) is also a solution of
(\ref{h14}). The five equations (\ref{h11}), (\ref{h12}) can be written as:%
\[
M\left(  x\left(  \theta\right)  \right)  u\left(  \theta\right)  =\left(
\begin{array}
[c]{c}%
0\\
v\left(  \theta\right)
\end{array}
\right)
\]
with $u=\left(  u^{i}\right)  ,1\leq i\leq5$, $\ v=\left(  v^{11}%
,v^{12},v^{22}\right)  $ and $M\left(  x\left(  \theta\right)  \right)  $ a
$5\times5$ matrix. These are no longer partial differential equations. If%
\begin{equation}
\det M\left(  x\left(  \theta\right)  \right)  =\det\left(  \frac{\partial
x}{\partial\theta_{1}},\frac{\partial x}{\partial\theta_{2}},\frac
{\partial^{2}x}{\partial\theta_{1}^{2}},\frac{\partial^{2}x}{\partial
\theta_{1}\partial\theta_{2}},\frac{\partial^{2}x}{\partial\theta_{2}^{2}%
}\right)  \neq0 \label{68}%
\end{equation}
they can be solved pointwise. Set $L\left(  x\left(  \theta\right)  \right)
:=M\left(  x\left(  \theta\right)  \right)  ^{-1}$, and denote by $M\left(
x\right)  $ and $L\left(  x\right)  $ the operators $u\left(  \theta\right)
\rightarrow M\left(  x\left(  \theta\right)  \right)  u\left(  \theta\right)
$ and $v\left(  \theta\right)  \rightarrow L\left(  x\left(  \theta\right)
\right)  v\left(  \theta\right)  $.

Since $x^{0}\in H^{\mu}$, with $\mu>3$, we have $x^{0}\in C^{2}$, so $M\left(
x^{0}\left(  \theta\right)  \right)  $ is well-defined and continuous with
respect to $\theta$. If $\det M\left(  x^{0}\left(  \theta\right)  \right)
\neq0$ for all $\theta\in\mathbb{T}_{2}$, then there will be some $R>0$ and
some $\varepsilon>0$ such that $\left\vert \det M\left(  x\left(
\theta\right)  \right)  \right\vert $ $\geq$ $\varepsilon$ for all $x$ with
$\left\Vert x-x^{0}\right\Vert _{\mu}\leq R$. So the operator $L\left(
x\right)  $ is a right-inverse of $DF\left(  x\right)  $ on $\left\Vert
x-x^{0}\right\Vert _{\mu}\leq R$, and we have the uniform estimates, valid on
$x^{0}+B_{\mu}\left(  R\right)  $ and $s\geq0$%

\begin{align*}
\left\Vert DF\left(  x\right)  u\right\Vert _{0}  &  \leq C_{0}\left\Vert
x\right\Vert _{1}\left\Vert u\right\Vert _{1}\\
\left\Vert DF\left(  x\right)  u\right\Vert _{s}  &  \leq C_{s}\left(
\left\Vert x\right\Vert _{s+1}\left\Vert u\right\Vert _{1}+\left\Vert
x\right\Vert _{1}\left\Vert u\right\Vert _{s+1}\right) \\
\left\Vert L\left(  x\right)  v\right\Vert _{0}  &  \leq C_{0}\left\Vert
x\right\Vert _{\mu}\left\Vert v\right\Vert _{0}\\
\left\Vert L\left(  x\right)  v\right\Vert _{s}  &  \leq C_{s}^{\prime}\left(
\left\Vert x\right\Vert _{\mu+s}\left\Vert v\right\Vert _{0}+\left\Vert
x\right\Vert _{\mu}\left\Vert v\right\Vert _{s}\right)
\end{align*}

This means that the tame estimate (\ref{tame}) is satisfied. However,
(\ref{tameinverse}) requires a proof. For this, we have to build a sequence of
projectors $\Pi_{N}$ satisfying the estimates (\ref{loss}) and (\ref{gain}).
For this, we use a multiresolution analysis of $L^{2}\left(  \mathbb{R}%
\right)  $ (see \cite{Meyer}). Recall that it is an increasing sequence
$V_{N},N\in\mathbb{Z},$ of closed subspaces of $L^{2}\left(  \mathbb{R}%
\right)  $ with the following properties:

\begin{itemize}
\item $\cap_{N=-\infty}^{N=+\infty}V_{N}=\left\{  0\right\}  $ and
$\cup_{N=-\infty}^{N=+\infty}V_{N}$ is dense in $L^{2}$

\item $u\left(  t\right)  \in V_{N}\Longleftrightarrow u\left(  2t\right)  \in
V_{N+1}$

\item $\forall k\in\mathbb{Z},\ \ u\left(  t\right)  \in V_{0}%
\Longleftrightarrow u\left(  t-k\right)  \in V_{0}$

\item there is a function $\varphi\left(  t\right)  \in V_{0}$ such that the
$\varphi\left(  t-k\right)  ,k\in\mathbb{Z}$, constitute a Riesz basis of
$L^{2}$.
\end{itemize}

It is known (\cite{Meyer}, Theorem III.8.3) that for every $r\geq0$ there is a
multiresolution analysis of $L^{2}\left(  \mathbb{R}\right)  $ such that:

\begin{itemize}
\item the $\varphi\left(  t-k\right)  ,k\in\mathbb{Z}$, constitute an
orthogonal basis of $V_{0}$

\item $\varphi\left(  t\right)  $ is $C^{r}$ and has compact support:\ there
is some $a$ (depending on $r$) such that $\left\vert t\right\vert \geq
a\Longrightarrow\varphi\left(  t\right)  =0$.
\end{itemize}

We choose $r$ so large that $C^{r}\subset H^{S}$. Set $\varphi_{N}\left(
t\right)  :=2^{N/2}\varphi\left(  2^{N}t\right)  $. For $N$ large enough, say
$N\geq N_{0}$, the function $\varphi_{N}$ has its support in $]-1/2,\ 1/2[$ so
we can consider it as a function on $\mathbb{R}/\mathbb{Z}$, and the
$\varphi_{N,k}\left(  \theta\right)  :=2^{N/2}\varphi\left(  2^{N}%
\theta-k\right)  $, for $0\leq k\leq2^{N}-1$, constitute an orthonormal basis
of $V_{N}$. In this way, we get a multiresolution analysis on $L^{2}\left(
\mathbb{R}/\mathbb{Z}\right)  $. Setting%
\begin{align*}
\Phi_{N,k}\left(  \theta\right)   &  =2^{N}\varphi\left(  2^{N}\theta
_{1}-k_{1}\right)  \varphi\left(  2^{N}\theta_{2}-k_{2}\right)  \\
E_{N} &  =\mathrm{Span}\left\{  \Phi_{N,k}\left(  \theta\right)
\ |\ k=\left(  k_{1},k_{2}\right)  ,\ 0\leq k_{1},k_{2}\leq2^{N}-1\right\}
\end{align*}
we get a multiresolution analysis of $L^{2}\left(  \mathbb{T}_{2}\right)  $,
and the $\Phi_{N,k}\left(  \theta\right)  $ are an orthonormal basis of
$E_{N}$.\ Finally, the $E_{N}^{5}$ constitute a multiresolution analysis of
$L^{2}\left(  \mathbb{T}_{2}\right)  ^{5}=$ $L^{2}\left(  \mathbb{T}%
_{2},\mathbb{R}^{5}\right)  $. Denote by $\Pi_{N}$ the orthogonal projection:%
\[
\left(  \Pi_{N}u\right)  ^{i}=\sum_{i,k}<\Phi_{N,k},u^{i}>\Phi_{N,k}%
\]

Introduce an orthonormal basis of wavelets associated with this
multiresolution analysis $\left(  E_{N}^{5}\right)  _{N\geq0}$ of
$L^{2}\left(  \mathbb{T}_{2},\mathbb{R}^{5}\right)  $. More precisely, the
$\Phi_{N_{0},k},0\leq k_{1},k_{2}\leq2^{N}-1$, span $E_{N_{0}}$, and one can
find a $C^{r}$ function $\Psi$ with compact support such that the $\Psi
_{N,k}=2^{N}\Psi\left(  2^{N}\theta_{1}-k_{1},2^{N}\theta_{2}-k_{2}\right)  $
span the orthogonal complement of $E_{N-1}$ in $E_{N}$. The $\Phi_{N_{0},k}$
and the $\Psi_{N,k}\ $for $N\geq N_{0}$ constitute an orthonormal basis for
$L^{2}\left(  \mathbb{T}_{2}\right)  $. We have, for $u\in L^{2}\left(
\mathbb{T}_{2}\right)  $%
\begin{align*}
u^{i}  &  =\sum_{k}\langle u^{i},\ \Phi_{N_{0},k}\rangle\Phi_{N_{0},k}%
+\sum_{N\geq N_{0},}\sum_{k}\langle u^{i},\ \Psi_{N,k}\rangle\Psi_{N,k}\\
\left\Vert u\right\Vert _{L^{2}}^{2}  &  =\sum_{i=1}^{5}\left(  \sum
_{k}\langle u^{i},\ \Phi_{N_{0},k}\rangle^{2}+\sum_{n\geq N_{0},}\sum
_{k}\langle u^{i},\ \Psi_{n,k}\rangle^{2}\right)
\end{align*}

It follows from the definition that, for $N\geq N_{0}$, we have:%
\begin{align*}
\left\Vert \Pi_{N}u\right\Vert ^{2}  &  =\sum_{i=1}^{5}\left(  \sum_{k}\langle
u^{i},\ \Phi_{N_{0},k}\rangle^{2}+\sum_{N_{0}\leq n\leq N,}\sum_{k}\langle
u^{i},\ \Psi_{n,k}\rangle^{2}\right) \\
\left\Vert \Pi_{N}u-u\right\Vert _{L^{2}}^{2}  &  =\sum_{i=1}^{5}\sum_{n\geq
N,}\sum_{k}\langle u^{i},\ \Psi_{n,k}\rangle^{2}%
\end{align*}

It is known (see \cite{Meyer} Theorem III.10.4) that:%
\[
C_{1}\left\Vert u\right\Vert _{H^{s}}^{2}\leq\sum_{k\in\left(  K_{N}\right)
^{5}}2^{2N_{0}s}\langle u,\ \Phi_{N_{0},k}\rangle^{2}+\sum_{n\geq N_{0},}%
\sum_{k\in\left(  K_{n}\right)  ^{5}}2^{2ns}\langle u,\ \Psi_{n,k}\rangle
^{2}\leq C_{2}\left\Vert u\right\Vert _{H^{s}}^{2}%
\]

If $v=\Pi_{N}u$, we must have $\langle v,\ \Psi_{n,k}\rangle=0$ for all $n\geq
N+1$, so that:
\[
C_{1}\left\Vert \Pi_{N}u\right\Vert _{H^{s}}^{2}\leq2^{2Ns}\left\Vert \Pi
_{N}u\right\Vert _{L^{2}}^{2}\leq2^{2Ns}\left\Vert u\right\Vert _{L^{2}}^{2}%
\]%
\begin{align*}
\left\Vert \Pi_{N}u-u\right\Vert _{L^{2}}^{2}  &  \leq2^{-2Ns}\sum_{n\geq
N,}\sum_{k\in\left(  K_{n}\right)  ^{5}}2^{2ns}\langle u,\ \Psi_{n,k}%
\rangle^{2}\\
&  \leq2^{-2Ns}C_{2}\left\Vert u\right\Vert _{H^{s}}^{2}%
\end{align*}

So estimates (\ref{loss}) and (\ref{gain}) have been proved.

Finally, we prove (\ref{tameinverse}). We have:%
\begin{align*}
\left(  \Pi_{N}u\right)  ^{i}\left(  \theta\right)   &  =\sum_{k}u_{k}%
^{i}\varphi_{N,k}\left(  \theta\right)  ,\ 1\leq i<5\\
\left(  M\left(  x\left(  \theta\right)  \right)  \Pi_{N}u\right)  ^{j}\left(
\theta\right)   &  =\sum_{k}\sum_{i}\varphi_{N,k}\left(  \theta\right)
\left(  M_{i}^{j}\left(  x\left(  \theta\right)  \right)  u_{k}^{i}\right)
\end{align*}

So $\Pi_{N}M\left(  x\right)  \Pi_{N}$ is a $\left(  2^{N}-1\right)
\times\left(  2^{N}-1\right)  $ matrix of $5\times5$ matrices $m_{k,k^{\prime
}}$, with:%
\[
m_{k,k^{\prime}}=\int_{\mathbb{T}_{2}}\varphi_{N,k}\left(  \theta\right)
\varphi_{N,k^{\prime}}\left(  \theta\right)  M\left(  x\left(  \theta\right)
\right)
\]

Looking at the supports of $\varphi_{N,k}$ and $\varphi_{N,k^{\prime}}$, we
see that there is a band along the diagonal outsides which $m_{kk^{\prime}}$
vanishes.
\begin{equation}
m_{k,k^{\prime}}=0\text{ if }\max_{i=1,2}\left\vert k_{i}-k_{i}^{\prime
}\right\vert >2^{1-N}a \label{ddd}%
\end{equation}

Choose some $\varepsilon>0$ and $N$ so large that $\max_{i=1,2}\left\vert
\theta-2^{-N}k_{i}\right\vert \leq2^{1-N}a$ implies that $\left\Vert M\left(
x\left(  \theta\right)  \right)  -M\left(  x\left(  2^{-N}k\right)  \right)
\right\Vert \leq\varepsilon$. Then:%
\[
\left\Vert m_{k,k^{\prime}}-\int_{\mathbb{T}_{2}}\varphi_{N,k}\left(
\theta\right)  \varphi_{N,k^{\prime}}\left(  \theta\right)  M\left(  x\left(
2^{-N}k\right)  \right)  d\theta\right\Vert \leq\varepsilon\int_{\mathbb{T}%
_{2}}\left\vert \varphi_{N,k}\left(  \theta\right)  \right\vert \left\vert
\varphi_{N,k^{\prime}}\left(  \theta\right)  \right\vert d\theta
\]

Using the fact that the system $\varphi_{N,k}$ is orthonormal, we get:%
\[
\left\Vert m_{k,k^{\prime}}-\delta_{k,k^{\prime}}M\left(  x\left(
2^{-N}k\right)  \right)  \right\Vert \leq\varepsilon\left(  \max
\varphi\right)  ^{2}%
\]

In addition, for every $k$, (\ref{ddd}) gives us at most $4a^{2}$ non-zero
values for $k^{\prime}$. So the matrix $m_{k,k^{\prime}}$ is a small
perturbation of the diagonal matrix $\Delta_{N}:=\delta_{k,k^{\prime}}M\left(
2^{-N}k\right)  ,\ k\in K_{N}$, which is invertible by (\ref{68}). More
precisely, we have:%
\[
\left(  \Pi_{N}M\left(  x\right)  \Pi_{N}\right)  ^{-1}=\left[  I+\Delta
_{N}^{-1}\left(  \Pi_{N}M\left(  x\right)  \Pi_{N}-\Delta_{N}\right)  \right]
^{-1}\Delta_{N}^{-1}%
\]
with $\left\Vert \Pi_{N}M\left(  x\right)  \Pi_{N}-\Delta_{N}\right\Vert
\leq\varepsilon4a^{2}\left(  \max\varphi\right)  ^{2}$ and $\left\Vert
\Delta_{N}^{-1}\right\Vert \leq\max_{\theta}\left\Vert M\left(  x\left(
\theta\right)  \right)  ^{-1}\right\Vert $. So $\left(  \Pi_{N}M\left(
x\right)  \Pi_{N}\right)  $ is invertible for $\varepsilon$ small enough, for
instance:%
\[
\varepsilon4a^{2}\left(  \max\varphi\right)  ^{2}\max_{\theta}\left\Vert
M\left(  x\left(  \theta\right)  \right)  ^{-1}\right\Vert <\frac{1}{2}%
\]
and we have:
\[
\left\Vert \left(  \Pi_{N}M\left(  x\right)  \Pi_{N}\right)  ^{-1}\right\Vert
\leq2\left\Vert \Delta_{N}^{-1}\right\Vert =2\left\Vert \min_{\theta}M\left(
x\left(  \theta\right)  \right)  ^{-1}\right\Vert
\]

Estimate (\ref{tameinverse}) then follows immediately

We now introduce the new spaces $X^{s}:=H^{s+\mu}$ and $Y^{s}:=H^{s+\mu-1}$.
We denote their norms by $\left\Vert x\right\Vert _{s}^{\ast}:=\left\Vert
x\right\Vert _{s+\mu}$ and $\left\Vert y\right\Vert _{s}^{\ast}:=\left\Vert
y\right\Vert _{s+\mu-1}$ respectively, so the above estimates become:%
\begin{align*}
\left\Vert DF\left(  x\right)  u\right\Vert _{s}^{\ast}  &  \leq C_{s}\left(
\left\Vert x\right\Vert _{s}^{\ast}\left\Vert u\right\Vert _{0}^{\ast
}+\left\Vert x\right\Vert _{0}^{\ast}\left\Vert u\right\Vert _{s}^{\ast
}\right)  \ \ \ \text{for }s\geq0\\
u\left\Vert L\left(  x\right)  v\right\Vert _{s}^{\ast}  &  \leq C_{s}%
^{\prime}\left(  \left\Vert x\right\Vert _{s+\mu}^{\ast}\left\Vert
v\right\Vert _{1}^{\ast}+\left\Vert x\right\Vert _{0}^{\ast}\left\Vert
v\right\Vert _{s+1}^{\ast}\right)  \text{ \ \ for }s\geq0
\end{align*}
and the range for $s$ becomes $0\leq s\leq S$ with $S=\sigma+\mu-1$.We have
proved that all the conditions of Corollary \ref{c11} are satisfied, with
$x^{0}\in H^{\mu}=X^{0}$, $g^{0}\in H^{\sigma}=Y^{\sigma-\mu+1}$ and
$S=\sigma-\mu+1$. Hence:

\begin{theorem}
Take any $\mu>3$. Suppose $x^{0}\in H^{\mu}$ and $g^{0}\in H^{\sigma}$ with
$\sigma\geq5\mu-1>14$. Suppose the determinant (\ref{68}) does not vanish. Set
$S:=\sigma-\mu+1$. Then, for any $\delta$ and $\alpha$ such that:%
\begin{align*}
\frac{\delta}{\mu}  &  \geq\varphi\left(  \frac{S}{\mu}\right) \\
\frac{\alpha}{\mu}  &  <\min\left\{  \frac{\delta}{\mu}-\varphi\left(
\frac{S}{\mu}\right)  ,\ \frac{S}{\mu}-\varphi^{-1}\left(  \frac{\delta}{\mu
}\right)  \right\}
\end{align*}
there is some $\rho\,>0$ and some $C>0$ such that, for any $g$ with
$\left\Vert g-g^{0}\right\Vert _{\delta+\mu-1}\leq\rho$ , there is $x\in
H^{\mu+\alpha}$ such that:
\begin{align*}
\left(  \frac{\partial x}{\partial\theta_{i}},\frac{\partial x}{\partial
\theta_{j}}\right)   &  =g_{i,j}\left(  \theta_{1},\theta_{2}\right) \\
\left\Vert x-x^{0}\right\Vert _{\mu}  &  \leq1\\
\Vert x-x^{0}\Vert_{\mu+\alpha}  &  \leq C\left\Vert g-g^{0}\right\Vert
_{\mu+\delta-1}%
\end{align*}

\end{theorem}

Moser \cite{Moser} finds that if $g,g^{0}\in C^{r+40}$ and $x^{0}\in C^{r}$
for some $r\geq2$, and if $\left\vert g-g^{0}\right\vert _{r}$ is sufficiently
small, then we can solve the problem. Although he made no effort to get
optimal differentiability assumptions, we note that our loss of regularity is
substantially smaller ($\sigma-\mu\geq4\mu-1>11$ instead of $40$).

Note that when $\mu\rightarrow3$, we have $\sigma\rightarrow14$ and
$\delta\rightarrow\infty$. In another direction:

\begin{corollary}
Suppose $x^{0}\in C^{\infty}$, $g^{0}\in C^{\infty}$, and the determinant
(\ref{68}) does not vanish. Then, for any $\delta>\mu>3$ and any
$\alpha<\delta-\mu$, there is some $\rho\,>0$ and some $C>0$ such that, for
any $g$ with $\left\Vert g-g^{0}\right\Vert _{\delta+\mu-1}\leq\rho$ , there
is some $x\in H^{\alpha+\mu}$ such that:
\begin{align*}
\left(  \frac{\partial x}{\partial\theta_{i}},\frac{\partial x}{\partial
\theta_{j}}\right)   &  =g_{i,j}\left(  \theta_{1},\theta_{2}\right) \\
\left\Vert x-x^{0}\right\Vert _{\mu}  &  \leq1\\
\Vert x-x^{0}\Vert_{\mu+\alpha}  &  \leq C\left\Vert g-g^{0}\right\Vert
_{\mu+\delta-1}%
\end{align*}

\end{corollary}

\end{document}